\documentclass{article}
\usepackage{arxnotes}
\let\wtilde\widetilde
\usepackage{url}

\makeop{HS}
\title{Triviality of arc closures and the local isomorphism problem}
\author{Devlin Mallory\thanks{The author was supported by NSF Graduate Research Fellowship grant DGE-1256260, as well as partially supported by NSF grant DMS-1701622.}}
\def\ac{^{\mathrm{ac}}}
\def\jc#1{^{#1-\mathrm{jc}}}
\def\asc{^{\mathrm{asc}}}
\def\jsc#1{^{#1-\mathrm{jsc}}}
\def\jscn{^{\mathrm{jsc}}}
\def\a{\mathfrak a}
\def\m{\mathfrak m}
\def\n{\mathfrak n}
\DeclareMathOperator{\Length}{length}

\theoremstyle{theorem}
\newtheorem{thm}{Theorem}
\numberwithin{thm}{section}
\newtheorem{lem}[thm]{Lemma}
\newtheorem{cor}[thm]{Corollary}
\newtheorem{prop}[thm]{Proposition}
\theoremstyle{definition}
\newtheorem{dfn}[thm]{Definition}
\newtheorem{exa}[thm]{Example}
\newtheorem{rem}[thm]{Remark}
\def\mN{\m_N}
\begin{document}
\maketitle
\abstract{\small We give an answer  in the ``geometric'' setting to a question of \cite{DEI} asking when local isomorphisms of $k$-schemes can be detected on the associated maps of local arc or jet schemes. In particular, we show that their ideal-closure operation $\a\mapsto \a\ac$ (the arc-closure) on a local $k$-algebra $(R,\m,L)$  is trivial when $R$ is Noetherian and $k\hookrightarrow L$ is separable, and thus that such a  germ $\Spec R$ has the (embedded) local isomorphism property.
}

\section{Introduction}

Let $k$ be any field.
Given a $k$-scheme $X$,  morphisms $\Spec k[t]/t^{\ell+1} \to X$ (i.e., $\ell$-jets) or $\Spec k[[t]]\to X$ (i.e., arcs)  are parametrized 
by schemes  $J_\ell(X)$ and $J_\infty(X)$, the $\ell$-jet and arc schemes. 
%Introduced initially by \cite{Nash} in the complex setting in connection with the divisors appearing on a resolution of singularities, 
The arc and jet schemes encapsulate a great deal of information about $X$. They're central to the theory of motivic integration, which allowed Kontsevich to show the birational invariance of the Hodge numbers of Calabi--Yau varieties \cite{Kontsevich} and since then has been applied  to the study of various motivic invariants (see, e.g., \cite{motivic1,motivic2}). In particular, this has led to connections between singularities of the minimal model program and arc schemes (see \cite{Mircea2}).
In a somewhat different direction, they are related further to singularities  through  the study of the Nash blow-up and Mather--Jacobian discrepancy (see 
\cite{IshiiReguera,NashBlow}).

There are morphisms $\pi_\infty\colon J_\infty(X)\to X$ and $\pi_\ell\colon J_\ell(X)\to X$, 
%given on $k$-points by sending an arc $\Spec k[[t]]\to X$ or or $l$-let  $\Spec k[t]/t^{l+1}\to X$ to the image of the closed point, i.e., to the composition $\Spec k\to \Spec k[[t]] \to X$ or $\Spec k\to \Spec k[t]/t^{l+1}\to X$.
given by sending an arc $\Spec k[[t]]\to X$ to the image of the closed point of $\Spec k[[t]]$ in $X$, and likewise for $\ell$-jets.
Given a point $x$ of $X$, one defines $J_\infty(X)_x:=\pi_\infty\inv(x)$ and $J_\ell(X)_x:=\pi_\ell\inv(x)$,  the arcs or $\ell$-jets based at $x\in X$. 
Given a morphism of $k$-schemes $f\colon X\to Y$, we obtain morphisms
$f_\ell\colon J_\ell(X)\to J_\ell(Y)$ and $f_\infty\colon  J_\infty(X)\to J_\infty(Y)$, defined on $k$-points by sending an arc $\Spec k[[t]]\to X$ on $X$ to the arc $\Spec k[[t]]\to X\to Y$ on $Y$;
these 
restrict to morphisms
$\bar f_\ell\colon J_\ell(X)_x\to J_\ell(Y)_{f(x)}$ and $\bar f_\infty\colon  J_\infty(X)_x\to J_\infty(Y)_{f(x)}$.

In 
\cite{DEI},
de Fernex, Ein, and Ishii 
 considered the question of how much local information about $f$ is captured by the morphisms $\bar f_\ell\colon J_\ell(X)_x\to J_\ell(Y)_{f(x)}$ or $\bar f_\infty\colon  J_\infty(X)_x\to J_\infty(Y)_{f(x)}$. More precisely, they asked the following question: 

\begin{quest}[local isomorphism problem]
If the morphisms $\bar f_\ell\colon J_\ell(X)_x\to J_\ell(Y)_{f(x)}$ are isomorphisms for all $\ell$ (including $\ell=\infty$), is $f$ an isomorphism at $x$, i.e., does $f$ induce an isomorphism of local rings $\O_{Y,f(x)}\to \O_{X,x}$? 
\end{quest}

The question is local on $X$ and $Y$, so we can restrict the setting to where $X,Y$ are spectra of local rings and $x,y=f(x)$ are the closed points; we call such a pair $(X,x)$ a germ.

The article \cite{DEI} also considers the following variant
(and shows that this is equivalent to the original question when $X$ is locally Noetherian):

\begin{quest}[embedded local isomorphism problem]
If we assume furthermore that $f$ is a closed embedding of germs, does the above question  have a positive answer?
\end{quest}

%For a given $k$-scheme $Y$ and $y\in Y$, if the (embedded) local isomorphism problem has a positive answer for all morphisms $X\to (Y,y)$, we say that the germ $(Y,y)$ has the (embedded) local isomorphism \emph{property}.

In order to understand the embedded version of the question, de Fernex, Ein, and Ishii introduce the arc closure, which is a closure operation $\a \mapsto \a \ac$ on ideals of a local $k$-algebra $R$ defined using the jet schemes of $\Spec R$. They then show that arc-closure of the zero ideal (the equality $(0)=(0)\ac$) for a ring $R$   is equivalent to a positive answer to the embedded local isomorphism problem for morphisms to $(\Spec R,\Spec(R/\m))$.
They furthermore give an example of a (non-Noetherian) $k$-algebra $R$ in which the zero ideal is not arc-closed, and thus in which the embedded local isomorphism property does not hold, suggesting that some restrictions on $R$ are necessary to ensure a positive answer.
%They also give a positive characterization of several families of local rings that do satisfy the embedded local isomorphism property.

In this paper, we show that this closure operation is trivial for Noetherian local $k$-algebras $(R,\m,L)$ for which $k\hookrightarrow L$ is separable, and thus that such germs have the embedded local isomorphism property:

\begin{thm}
If $(R,\m,L)$ is a Noetherian local $k$-algebra with residue field $L$, $k\hookrightarrow L$ is separable, and $\a $ is a proper ideal of $R$, then $\a \ac =\a $.
\end{thm}

In particular, this holds when $\Char k=0$, when $k$ is perfect, or when $L=k$, and thus holds in the cases of primary geometric interest.

\begin{cor}
Such germs have the embedded local isomorphism property.
\end{cor}

The strategy is relatively simple:
We proceed by reducing first to showing that $\a \ac=\a $ for $\a $ an $\m$-primary ideal, and then by induction on the length of $R/\a $ to the case where $R/\a $ is a Gorenstein Artinian local $k$-algebra. At that point, 
 we obtain an inclusion of $R'$-modules $R/\a \hookrightarrow R'$ via the Matlis dual (\emph{not} an inclusion of rings!)
for a suitable \emph{graded} Gorenstein Artinian local $k$-algebra $R'$.
This step uses the Cohen structure theorem, and requires that $R/\a $ has a coefficient field $L_0\cong L$ containing $k$, which is where the assumption on separability of $k\hookrightarrow L$ comes in.
 This inclusion of modules necessitates the introduction and analysis of an arc-closure operation defined on submodules of modules; once a few elementary properties are shown, we use that the arc-closedness of the zero ideal of $R'$, as shown in \cite[Theorem~5.8(a)]{DEI}, to conclude that the zero ideal of $R/\a $ must be arc-closed as well.

The organization of the paper is as follows: In 
Section~\ref{arcs} we review the basic definitions of arc and jet schemes and the local isomorphism problem, and  
in Section~\ref{closures} we recall the definition of arc and jet closures and their basic properties from \cite{DEI}.
In 
Section~\ref{modules} we generalize the definition of arc and jet closures to 
closures of
submodules of modules and prove some elementary properties about these operations under module maps and restrictions of scalars along a ring quotient.
Section~\ref{proof} contains the core of our proof, and
Section~\ref{questions} has a few observations on further questions on the subject.

%As mentioned, then, if we restrict the source of the morphism to be (locally) Noetherian, this is equivalent to having the local isomorphism property.
\subsection*{Acknowledgements}
I am much indebted to the work of de Fernex, Ein, and Ishii on the local isomorphism problem, and thank them for their lovely reformulation of the problem in \cite{DEI}.
I also would like to thank Monica Lewis for fruitful conversations regarding closure operations, and 
my advisor
Mircea Musta\c{t}\u{a} for helpful discussions and advice.
I am grateful as well for very helpful feedback on previous versions of this paper
from
 Mircea Musta\c{t}\u{a} and Ilya Smirnov.

\section{Arc and jet schemes}
\label{arcs}

We recall the definition of arc and jet schemes; for a comprehensive treatment see, e.g., \cite{Vojta}.
Let $X$ be a scheme over $k$, and for each $\ell \in \N$ consider the functor from $k$-schemes to sets 
$$T\mapsto \Hom\bigl(T\times_k \Spec(k[t]/t^{\ell+1}),X\bigr).$$ 

One can show that there is a $k$-scheme $J_\ell(X)$, the scheme of $l$-jets of $X$, representing this functor, i.e., such that
$$
\Hom\bigl(T\times_k \Spec(k[t]/t^{\ell+1}), X\bigr) =\Hom(T,J_\ell(X));
$$
in particular, $k$-points of $J_\ell(X)$ correspond to maps $\Spec k[t]/t^{\ell+1}\to X$.
Moreover, if $X$ is finite-type over $k$ then so is $J_\ell(X)$.

The quotient maps 
$$k[t]/t^{\ell+1}\to k[t]/t^{\ell'+1}$$ for $\ell'<\ell$ induce morphisms $$\pi_{\ell,\ell'}\colon J_\ell(X)\to J_{\ell'}(X).$$
It's easily verified that these maps are affine, 
and thus the inverse limit  over the system $\set{J_\ell(X)\to J_{\ell'}(X):\ell>\ell'}$ exists in the category of $k$-schemes. We denote this limit by $J_\infty(X)$, and call it the arc scheme of~$X$ (note that $J_\infty(X)$ will not be finite-type over $k$ in general). 
$J_\infty(X)$ is equipped with canonical truncation maps $\pi_{\infty,\ell}\colon J_\infty(X)\to J_l(X)$ for any $\ell$, arising from the quotient morphism
$$k[[t]]\to k[t]/t^{\ell+1}.$$

\begin{rem}
One can check that if $k\hookrightarrow L$ is a field extension then
$$
\let\big\relax
\Hom\bigl(\Spec(L[[t]]), X\bigr)
=
\Hom(\Spec (L), J_\infty(X)).
$$
In fact,
by \cite{Bhargav}
it is true (but highly nonelementary)
that
 if $X$ is quasicompact and quasiseparated over $k$ and
$S$ is a $k$-algebra 
then
$$
\Hom\bigl(\Spec S\times_k \Spec(k[[t]]/t^{\ell+1}), X\bigr) =\Hom(\Spec S,J_\ell(X)),
$$
but we do not use this in the following.
\end{rem}

In the following, we use $\ell$ to denote an element of $\N\cup \set\infty$, and write $k[[t]]/t^{\ell+1}$ to mean either $k[[t]]/t^{\ell+1}=k[t]/t^{\ell+1}$ when $\ell$ is finite or $k[[t]]$ when $\ell=\infty$.

Recall that for any $\ell$ we have truncation maps $\pi_{\ell,0}\colon  J_\ell(X)\to J_0(X)=X$, which we denote simply by $\pi_\ell$; 
at the level of $k$-points,
this just sends an arc $\Spec k[[t]]/t^{\ell+1} \to \Spec X$ to
$$\Spec k\to \Spec k[[t]]/t^{\ell+1} \to X,$$ i.e., to the image of the closed point of $\Spec k[[t]]/t^{\ell+1}$.
For a point $x\in X$, not necessarily closed,
we write $J_\ell(X)_x$ for $\pi_\ell\inv(x)$, the fiber over $x$.

Given a morphism $f\colon X\to Y$ of $k$-schemes, for any morphism 
$$T\times_k \Spec(k[[t]]/t^{\ell+1})\to X$$
 we obtain a morphism
$$T\times_k \Spec(k[[t]]/t^{\ell+1})\to X\to Y,$$
and 
by functoriality we obtain morphisms $f_\ell\colon J_\ell(X)\to J_\ell(Y)$ for all $\ell$.
Furthermore, it's clear that for $x\in X$ these morphisms restrict to morphisms
$ \bar f_\ell\colon J_\ell(X)_x\to J_\ell(Y)_{f(x)} $.

We now have the language to state the motivating questions of \cite{DEI}:

\begin{quest}[local isomorphism problem]
Given a map $f\colon X\to Y$ and $x\in X$,
if all the morphisms $\bar f_\ell\colon  J_\ell(X)_x\to J_\ell(Y)_{f(x)}$ are isomorphisms (including $\ell=\infty$), is $f$ a local isomorphism at $x$, i.e., does $f$ induce an isomorphism of local rings $\O_{Y,f(x)}\to \O_{X,x}$? 
\end{quest}
%(See Remark~\ref{finite} for the relation between the conditions that $\bar f_\ell$ be isomorphisms for \emph{all} $\ell$ and for all \emph{finite} $\ell$.)

\begin{quest}[embedded local isomorphism problem]
If we assume furthermore that $f$ is a closed embedding, does the above question  have a positive answer?
\end{quest}

As remarked previously, these questions are local on source and target of the morphism, so we may assume $X$ and $Y$ are spectra of local rings with closed points $x,y$ respectively; we refer to such a pair $(X,x)$ or $(Y,y)$ as a germ; if $(R,\m)$ is a local ring, we'll refer to $(\Spec R,\Spec (R/\m))$ simply as the germ $\Spec R$ when no confusion will occur.
We say a germ $(Y,y)$ has the  local isomorphism property (respectively, the embedded local isomorphism property) if the  local isomorphism problem (respectively, the embedded local isomorphism problem) has an affirmative answer for all maps of germs $(X,x)\to (Y,y)$.

\begin{rem}
As noted in \cite[Proposition 2.6, Lemma 2.7]{DEI}, if $Y$ has the embedded local isomorphism property then the local isomorphism problem has an affirmative answer for maps $(X,x)\to (Y,y)$ with $X$ locally Noetherian; thus, for most cases of geometric interest it suffices to consider just the embedded form of the problem.
\end{rem}

Now, we recall the local construction of the arc and jet schemes of an affine scheme $\Spec R$ (the arc and jet schemes are obtained by simply gluing the construction over affine charts).
Given a $k$-algebra $R$ (not necessarily Noetherian or local), and $\ell\in \N\cup\set\infty$, we write $R_\ell$ for the ring of Hasse--Schmidt derivations $\HS^\ell_{R/k}$.
If $\ell < \infty$,
for any $k$-scheme $T$ we have
$$\Hom(T,\Spec R_\ell)=\Hom(T\times_k \Spec(k[[t]]/t^{\ell+1}),X),
$$
so that
 $\Spec R_\ell$ is canonically isomorphic to $J_\ell(\Spec R)$; moreover, one can check that $\Spec R_\infty \cong J_\infty(\Spec R)$.

For an ideal $I$ of $R$ we write $I_\ell$ for the ideal of $R_\ell$ generated by $D_i(f)$ for $0\leq i< \ell+1$ and $f\in I$, where $D_i$ are the universal Hasse--Schmidt derivations $R\to R_\ell$.

\begin{exa}
If 
$R=k[x_1,\dots,x_n]$
then
$$R_\ell=k\bigl[x_i^{(j)}: i=1,\dots,n, 0\leq j \leq \ell\bigr];$$ 
one can think of a point $(a_i^{(j)})$ of $\Spec R_\ell\cong \A^{n(\ell+1)}$ as parametrizing the arc
$$
\Spec k[t]/t^{\ell+1} \to \A^n,\quad t\mapsto \biggl(
\sum a_1^{(i)} t^i,
\dots,
\sum a_n^{(i)} t^i\biggr).
$$
The Hasse--Schmidt derivations can be defined 
by setting
$$
D_i(x_j)=x_j^{(i)}
$$
and extending via the Leibniz rule $D_m(fg)=\sum_{i+j=m} D_i(f)D_j(g)$, so, e.g., if we write $x=x_1$, $y=x_2$, we have
$$D_2(xy)=D_2(x)D_0(y)+2D_1(x)D_1(y)+D_0(x)D_2(y)=
x^{(2)}y^{(0)}
+
2x^{(1)}y^{(1)}
+
x^{(0)}y^{(2)}.
$$
\end{exa}

The projection maps $\pi_{\ell,\ell'}$ for $\ell>\ell'$ give
ring maps $R_{\ell'} \to R_{\ell}$, and in particular
%, arising as follows: the projection
%$k[t]/t^{k}\to k[t]/t^l$
%allows us, for any $X$, to regard a map 
%$$X\times_k \Spec k[t]/t^k \to \Spec R$$ as a map
%$$X\times_k \Spec k[t]/t^l \to \Spec R$$ via composition with the natural map $$X\times_k \Spec k[t]/t^l\to X\times_k \Spec k[t]/t^k;$$
%this then yields a natural transformation of functors
%$$
%\Hom(-,\Spec R_k) \to \Hom(-,\Spec R_\ell)
%$$
%and thus a map $R_\ell\to R_k$.
an inclusion $R=R_0 \hookrightarrow R_\ell$ for any $\ell$.

\section{Arc and jet closures}
\label{closures}
Now, say $(R,\m)$ is a local $k$-algebra,
and write $\m R_\ell$ for the expansion of $\m\subset R$ to $R_\ell$ under the ring map $R\to R_\ell$.
The following definition is key to the reduction in \cite{DEI} of the embedded local isomorphism problem to a ring-theoretic question:

\begin{dfn}[\cite{DEI}]
\label{defn}
For an ideal $\a $ of $R$ and $\ell<\infty$, define the $\ell$-jet closure of $\a$ as 
$$\a \jc \ell=\bigl( f\in R: (f)_\ell \subset \a_\ell + \m R_\ell \bigr),$$
and
for $\ell=\infty$, define the arc closure of $\a$ as
$$\a \ac =\bigl( f\in R: (f)_\infty \subset \a_\infty + \m R_\infty \bigr).$$
\end{dfn}

Geometrically, $\a \jc \ell$ is the largest ideal of $R$ whose higher differentials define the same closed subscheme in  $J_\ell(\Spec R)_{\Spec R/\m}$ (the fiber over the closed point of $R$) as that defined by the higher differentials of $\a $.

\begin{exa}
It's immediately seen that the $\a \jc \ell$ are nontrivial closure operations; for example, by the Leibnitz rule it's easily seen that if $f\in \m^{\ell+1}$ then $D_\ell(f)\in \m R_\ell$, so that $\a +\m^{\ell+1}\subset \a \jc \ell$. For an example showing that this is in general a proper inclusion, see \cite[Example~3.11]{DEI}.
\end{exa}

The following shows that we can compute these closures in the quotient ring $R/\a $:

\begin{lem}[{\cite[Lemma 3.2]{DEI}}]
\label{hered}
Let $\a \subset \m$.
For all $\ell$, if $\pi\colon  R\to R/\a $, then
$$
\a \jc \ell = \pi\inv((0_{R/\a})\jc \ell),
$$
and similarly 
$$
\a \ac = \pi\inv((0_{R/\a})\ac ).
$$
\end{lem}

Thus it suffices to know how to compute the arc or $\ell$-jet closure of the zero ideal, for which there is a nice interpretation in terms of the ``universal'' $\ell$-jet:
the identity morphism $\Spec R_\ell\to \Spec R_\ell$ corresponds to the ``universal'' $\ell$-jet $(\Spec R_\ell)\times_k \Spec(k[t]/t^{\ell+1})\to \Spec R$, given by the ring map
$$
\mu_R\colon  R\to R_\ell[t]/t^{\ell+1};
$$
by composing with the quotient map
$$
R_\ell[t]/t^{\ell+1} \to 
%R_\ell[t]/t^{\ell+1}/ m R_\ell[t]/t^{\ell+1} =
\bigl(R_\ell/\m R_\ell\bigr)[t]/t^{\ell+1},
$$
we obtain a map
$$
\l_\ell\colon  R\to \bigl(R_\ell/\m R_\ell\bigr)[t]/t^{\ell+1}.
$$

The following statement is now clear by definition:

\begin{lem}[{\cite[Lemma~3.3]{DEI}}]
$(0_R)\jc \ell = \ker \l_\ell$
and
$(0_R)\ac  = \ker \l_\infty$.
\end{lem}

\begin{exa}
In the case
$R=k[x_1,\dots,x_n]_{(x_1,\dots,x_n)}$
the universal $\ell$-jet
$R\to R_\ell[t]/t^{\ell+1}$
sends
$$
x_i\mapsto x_i^{(0)} + x_i^{(1)}t 
+
x_i^{(2)}t^2
+ \dots+ x_i^{(\ell)}t^\ell.
$$
The ideal $\m R_\ell$ is just $\bigl(x_i^{(0)}:i=1,\dots,\ell\bigr)$; note that this is \emph{not} the expansion of $\m$ under the universal $\ell$-jet.
\end{exa}

The following
result linking $\a \jc \ell $ and $\a \ac$ is key to our proof below:

\begin{prop}[{\cite[Proposition~3.12]{DEI}}]
\label{comp}
$\bigcap_{\ell \geq 0} \a \jc \ell= \a \ac$.
\end{prop}

The geometric interpretation following Definition~\ref{defn} makes clear the motivation for this closure operation:

\begin{prop}[{\cite[Proposition~5.1]{DEI}}]
Let  $R$ be a local $k$-algebra.
The germ 
$\Spec R$ has the embedded local isomorphism property if and only if $(0_R)\ac = 0$.
\end{prop}

\begin{proof}
This is essentially by definition: a closed embedding $X\to \Spec R$ corresponds to a quotient $R\to R/\a $ for some $\a \subset R$;  this is an isomorphism of schemes if and only if $\a =0$, and it induces an isomorphism on the fibers of the jet schemes over the closed point if and only if $(\a )_\infty \subset \m R_\infty$ if and only if $\a \subset (0)\ac$.
\end{proof}

\begin{rem}
\label{finite}
  It is observed in \cite{DEI} that
the two preceding propositions imply that
 it is redundant in the statement of the embedded local isomorphism problem to ask for $\bar f_\ell$ to be an isomorphism for all $\ell\in \N\cup\set{\infty}$: $\bar f_\infty$ is an isomorphism if and only if $\bar f_\ell$ is an isomorphism for all finite $\ell$.
\end{rem}

\begin{rem}
In the non-Noetherian setting,
Proposition~5.4 of \cite{DEI} provides an example  of an ideal $\a $ inside a power series ring in infinitely many variables such that $\a \ac \neq \a$; this is proved via the observation that $\a \jc \ell \supset \a  +\m^{\ell+1}$, and then giving an explicit element contained in $\a +\m^{\ell+1}$ for all $\ell$.

We remark here that this situation may in some sense be typical, at least for certain classes of non-Noetherian rings: if $(R,\m)$ is a non-Noetherian valuation ring, then one has $\m=\m^2=\cdots$ (see, for example, \cite[Exercise~6.29]{SH}). Thus, for any ideal $\a \subset \m$, including the zero ideal, we have $\a +\m^{\ell}=\m$ for all~$\ell$, and thus $\a \ac=\m$.
\end{rem}

The last result we need from \cite{DEI} is that says that a graded $k$-algebra has arc-closed zero ideal; we'll write $R_{[i]}$ for the $i$-th graded piece of a graded ring $R$ to avoid confusion with the jet schemes $R_\ell$.

\begin{thm}[{\cite[Theorem~5.8(a)]{DEI}}]
\label{graded}
Let $(R,\m)$ be a local $k$-algebra with $\N$-grading such that $\m=\bigoplus_{i\geq 1} R_{[i]}$.
Then the zero ideal of $R$ is arc-closed.
\end{thm}

\begin{rem}
The hypotheses do not demand that $k$ be all of $R_{[0]}$ (which is the residue field of $R$); this is important in our application later.
\end{rem}

We recall their proof here for ease of reference:

\begin{proof}
We construct an explicit arc using the data of the grading:
define an arc $\rho\colon R\to R[[t]]$ by
sending a homogeneous element $f\in R_{[i]}$
to
$f t^i $.
It's immediate that $\rho$ is injective. By universality of the arc $R\to R_\infty[[t]]$ we get a map $\phi\colon R_\infty\to R$, inducing a map $\wtilde\phi\colon R_\infty[[t]]\to R[[t]]$ making the following diagram commute:
$$\begin{tikzcd}
  & R_\infty[[t]] \ar[d,"\wtilde \phi"]\\
%%\ar[r,two heads]   (R_\infty/mR_\infty)[[t]]\ar[ld,dashed]\\
R\ar[r,"\rho"',hook]\ar[ru,"\mu_R"] & R[[t]]
\end{tikzcd}$$
Now, observe that for $f\in \m$ we have 
$$\rho(f)=\wtilde\phi(\mu_R(f))=\wtilde \phi\bigl(d_0(f)+d_1(f)t+\cdots\bigr)
=\phi(d_0(f)) +\phi(d_1(f))t+\cdots.
$$
Since $\rho(f)\in tR[[t]]$, however, we must have that $\phi(d_0(f))=0$ for all $f\in \m$, and thus $\wtilde \phi$ factors through $R_\infty/\m R_\infty[[t]]$, yielding a commutative diagram
$$\begin{tikzcd}
  & R_\infty[[t]] \ar[d,"\wtilde \phi"]
\ar[r,two heads]  & (R_\infty/\m R_\infty)[[t]]\ar[ld,dashed]\\
R\ar[r,"\rho"',hook]\ar[ru,"\mu_R"] & R[[t]]
\end{tikzcd}$$
Thus, we must have that the composite map $\l_R\colon R\to R_\infty[[t]]\to (R_\infty/\m R_\infty)[[t]]$ is injective since $\rho $ is, and so $(0)\ac=0$.
\end{proof}

We also require the following persistence statement:

\begin{lem}
\label{ringpersistence}
Arc closures of ideals are persistent under local ring homomorphisms; that is, if 
$(R,\m)$ and $(S,\n)$ are local rings and
$\phi\colon R\to S$  a local homomorphism, and $\a \subset R$, then 
$\phi(\a\ac)\subset (\phi(\a)S)\ac$.
\end{lem}

\begin{proof}
First, note that we have a commutative diagram
$$\begin{tikzcd}
R\ar[r,"\phi"] \ar[d,"\pi"]& S\ar[d,"\pi'"]\\
R/\a \ar[r,"\wtilde \phi"] & S/\a S
\end{tikzcd}$$
By Lemma~\ref{hered},
elements $r\in \a\ac$ are precisely the elements of $R$ such that $\pi(r)$ lies in the arc-closure of $(0)$ in $R/\a$; thus if we can show persistence for the map $\wtilde \phi$ we have that $\pi'(\phi(r))=\wtilde\phi(\pi(r))$ lies in the arc closure of 0 in $S/\a S$, and thus applying Lemma~\ref{hered} again we have $\phi(r) \in (\phi(\a)S)\ac$. Thus, we may assume that $\a =(0)$.

We have an $(S_\infty/\n S_\infty)$-arc from $R$, i.e., the map $R\to S \to (S_\infty/\n S_\infty)[[t]]$; by universality of the arc $R\to R_\infty[[t]]$ this induces a ring map 
$R_\infty \to S_\infty/\n S_\infty$.
Since 
$\phi$ is local,  this descends to a ring map $R_\infty/\m R_\infty \to S_\infty/\n S_\infty$, and thus we have a ring map 
$$(R_\infty/\m R_\infty)[[t]] \to (S_\infty/\n S_\infty)[[t]],$$
fitting into the 
commutative diagram
$$\begin{tikzcd}
R\ar[d,"\l_R"]\ar[r,"\phi"] & S\ar[d,"\l_S"]\\
(R_\infty/\m R_\infty)  [[t]]
\ar[r,]
& (S_\infty/\n S_\infty) [[t]]
\end{tikzcd}$$
Commutativity of this diagram then implies immediately that since $(0_R)\ac=\ker\l_R$, we have  $\l_S\bigl(\phi((0_R)\ac)\bigr)=0$, so that $\phi((0_R)\ac)\subset (0_S)\ac$, yielding the result.
\end{proof}

\section{Arc closures of submodules}
\label{modules}
The key to our proof is to introduce the notion of arc-closure of an $R$-module:

\begin{dfn}
For an $R$-module $M$, define
$$
(0_M)\ac = \ker\biggl(
 M\xra{1_M\otimes \l_\infty} M\otimes_R \bigl(R_\infty/\m R_\infty\bigr)[[t]]
\biggr).
$$
For a submodule $N\subset M$, define $(N)_M\ac=\pi_N\inv\bigl((0_{M/N})\ac\bigr)$, where $\pi_N\colon M\to M/N$.
\end{dfn}

\begin{lem}
\label{persistence}
Arc closures of $R$-submodules are persistent under $R$-linear maps; that is, if $N\subset M$ is a submodule and $\phi\colon M\to M'$ is an $R$-module map, then $\phi((N)_M\ac)\subset (\phi(N))_{M'}\ac$.
%%%ADDED \PHI TO LHS 3/13/19
\end{lem}

\begin{proof}
Considering the commutative diagram
$$\begin{tikzcd}
M \ar[r,"\phi"]\ar[d,"\pi"] &  M'\ar[d,"\pi'"]\\
M/N \ar[r,"\bar\phi"] & M'/\phi(N)
\end{tikzcd}$$
we see that $m\in (N)_M\ac$ exactly when $\pi(m)\in (0)_{M/N}\ac$, and thus it suffices to show persistence under $\bar \phi$ to obtain it for $\phi$, i.e., it suffices to show persistence of arc closure of the zero submodule.

In this case, we have a commutative diagram
$$\begin{tikzcd}
M\ar[d]\ar[r,"\phi"] & M'\ar[d]\\
M\otimes \bigl(R_\infty/\m R_\infty\bigr)[t]/t^{\ell+1}\ar[r] &
M'\otimes \bigl(R_\infty/\m R_\infty\bigr)[t]/t^{\ell+1}
\end{tikzcd}$$
Note that $m \in M$ lies in the arc closure of 0 exactly when it's in the kernel of the left vertical map; when this occurs, commutativity of the diagram immediately implies that $\phi(m)$ is in the kernel of the right vertical map, so that $\phi(m)\in (0)_{M'}\ac$.
\end{proof}

We also need a comparison for closures as $R$-modules versus $R/I$-modules:

\begin{lem}
\label{comparison}
Let $R$ be a local ring and $I$ an ideal.
Let $M$ be an $R/I$-module, $N\subset M$ an $R/I$-submodule.
Then $$(N)_M\ac \subset ({}_RN)_{{}_R M}\ac,$$ 
where the right side is the closure of $N$ viewed as an $R$-submodule of the $R$-module $M$.
\end{lem}

In fact, we'll need this result only for the arc closure of 0 in $R/I$ itself, but we present the proof in the general case:

\begin{proof}
It suffices to show this for $N=0$, since the quotient map $M\to M/N$ is the same whether viewed as an $R$-module map or an $R/I$-module map.
Writing
$$
\eqalign{
\l_R&\colon  R\to (R_\infty/\m R_\infty)[[t]],\cr
\l_{R/I}&\colon  R/I \to ((R/I)_\infty/\m (R/I)_\infty)[[t]],
}$$
we have a commutative diagram of ring maps
%$$\begin{tikzcd}
%R/I\otimes _R R \ar[rr,"\id_{R/I}\otimes_R \l_R"] && R/I\otimes_R (R_\infty/mR_\infty)[[t]] \\
%R/I \ar[u,no head, double line] \ar[rr,"\l_{R/I}"]&& ((R/I)_\infty/m(R/I)_\infty)[[t]]  \ar[u,dashed]
%\end{tikzcd}$$
$$\begin{tikzcd}
&& R/I\otimes_R (R_\infty/\m R_\infty)[[t]] 
\\
R/I 
\ar[rru,"\id_{R/I}\otimes_R \l_R"]
 \ar[rr,"\l_{R/I}"']&& ((R/I)_\infty/\m (R/I)_\infty)[[t]]  \ar[u,dashed]
\end{tikzcd}$$
where the right vertical side is induced by the universality of the arc $\l_{R/I}$.
Tensoring over $R$ with $M$, we obtain
$$\begin{tikzcd}
 &&& 
M\otimes_R R/I\otimes_R (R_\infty/\m R_\infty)[[t]] 
\ar[r,no head, double line]
&
M\otimes_R (R_\infty/\m R_\infty)[[t]] 
\\
M \ar[rrru,"\id_M \otimes_{R}\l_R"] \ar[rrr,"\id_ M\otimes_{R}\l_{R/I}"']&&&
M\otimes_R  ((R/I)_\infty/\m (R/I)_\infty)[[t]] 
  \ar[u,dashed]\ar[r,no head, double line]
&
M\otimes_{R/I}
((R/I)_\infty/\m (R/I)_\infty)[[t]] 
\end{tikzcd}$$
Thus we see that since  
$$\id_M\otimes_R\, \l_R\colon M\to M\otimes_R (R_\infty/\m R_\infty)[[t]]$$ factors through  
$$
\id_M\otimes_{R}\,\l_{R/I}=
\id_M\otimes_{R/I}\,\l_{R/I}
\colon M\to M\otimes _{R/I} ((R/I)_\infty/\m (R/I)_\infty)[[t]],
$$ 
we must have that
$$
\underbrace{\ker\bigl(
\id_M\otimes_{R/I}\,\l_{R/I}
\bigr)}_{(N)\ac_M}\subset 
\underbrace{\ker\bigl(
\id_M\otimes_R\, \l_R
\bigr)}_
{({}_RN)_{{}_R M}\ac}
,
$$
yielding the result.
\end{proof}

\section{The main result}
\label{proof}

Given a local $k$-algebra $(R,\m,L)$ with residue field $L$, we say $L$ is separable over $k$ to mean that the field extension $k\subset R\to R/\m \cong L$ is separable (not necessarily algebraic).

\begin{thm}\label{main}
Let $(R,\m)$ be a Noetherian local $k$-algebra 
with residue field $L$ separable over $k$, 
and $\a $ a proper ideal of $R$. Then $\a \ac = \a $.
\end{thm}

As stated in the Introduction, the condition on separability of $k\hookrightarrow L$ is just to ensure that for a complete local $k$-algebra with residue field $L$ we may choose a coefficient field containing $k$; this is sufficient but not necessary, as can be seen by taking $L$ to be an inseparable extension of $k$ and setting $R=L[[x]]$;  $k\subset L$ is inseparable, but clearly $R$ has a coefficient field containing $k$.

We note that the assumption on $k\hookrightarrow L$ is satisfied in particular when $k$ has characteristic 0 or is perfect of positive characteristic, or when $k=L$, and thus in the primary case of geometric interest for the embedded local isomorphism question.

\begin{proof}[Proof of Theorem~\ref{main}]
The first step is to reduce to the case where $\a $ is $\m $-primary:

\begin{lem}
Let $(R,\m)$ be a local $k$-algebra. Then
$$
\a \jc \ell  = \bigcap_n \,(\a +\m ^n)\jc \ell
$$
for any $\ell$, and thus
$$
\a \ac =\bigcap_n \,(\a +\m ^n)\ac.
$$
\end{lem}

\begin{proof}
The second statement follows from the first due to the equality $\bigcap_\ell \a \jc \ell = \a \ac$ (see Proposition~\ref{comp}), since then
$$
\bigcap_n \,(\a +\m ^n)\ac =\bigcap_n \,\bigcap_\ell (\a +\m ^n)\jc \ell= 
\bigcap_\ell \,\bigcap_n (\a +\m ^n)\jc \ell= \bigcap_\ell \a  \jc \ell=\a \ac.
$$
Fix $\ell$. 
Clearly $\a \jc \ell \subset (\a +\m ^n)\jc \ell $ for all $n$ by monotonicity of the closure operation.
To see the other inclusion, note $(\a +\m ^n)\jc \ell=\a \jc \ell +(\m ^n)\jc \ell$. For $n>l$, though, the Leibniz rule says that $(\m ^n)\jc \ell\subset \m R_n$, so that $$\a \jc \ell+\m R_\ell=(\a +\m ^n)\jc \ell+\m R_\ell.$$ 
Thus $\a \jc \ell =(\a +\m ^n)\jc \ell$ for $n>\ell$, and the result follows.
\end{proof}

In the Noetherian case, then,
to see that $\a \ac=\a $
it suffices to show that $(\a +\m ^n)\ac= \a +\m ^n$, since then 
$$
\a\ac = \bigcap_n(\a+\m^n)\ac  = \bigcap_n \a+\m^n=\a,
$$
where the last equality follows by
Krull's intersection theorem.
 Equivalently by Lemma~\ref{hered}, we must show that the zero ideal is closed in any Artinian local $k$-algebra.
By induction, we may reduce further to the case of a Gorenstein Artinian local $k$-algebra:

\begin{lem}
If $(0_R)\ac=0_R$ for any Gorenstein Artinian local $k$-algebra $R$, the same is true for any Artinian local $k$-algebra.
\end{lem}

\begin{proof}
We induct on  $\Length(R)$. Say $f\in(0_R)\ac$.  
\smallbreak\noindent\emph{Case 1:}
Say there is $g\in \Soc R$ with $f\notin (g)=L\cdot g$, and consider the map $\pi\colon R\to R/(g)$.
Then by persistence of arc-closure under ring maps (Lemma~\ref{ringpersistence}) we have that $\pi\bigl((0_{R})\ac\bigr)\subset (0_{R/(g)})\ac$; since $g\in \Soc R$, though, we have $\Length(R/(g))=\Length(R)-1$, and thus by induction we know $(0_{R/(g)})\ac=0_{R/(g)}$. But then $\pi(f)=0$, so $f\in (g)$, contradicting our assumption, and thus $(0_R)\ac=0$.

\smallbreak\noindent\emph{Case 2:} There is no such $g\in \Soc R$, in which case we must have that $f$ itself generates the socle  of~$R$, and thus $R$ must be Gorenstein. In this case though $f=0$, by the assumption of the lemma.
\end{proof}

We're thus reduced to showing the zero ideal is arc-closed in a Gorenstein Artinian local $k$-algebra $R$ with residue field $L$ (since taking the quotient by an $\m $-primary ideal didn't change the residue field). 

By our assumption that $k\hookrightarrow L$ is separable and $R$ is an Artinian (hence complete) $k$-algebra with residue field $L$, there is a coefficient field $L_0\cong L$ contained in $ R$ containing $k$ (see \cite[Theorem 28.3]{Matsumura}).
By the Cohen structure theorem, such a ring $R$ can be written 
 as the quotient of $S=L_0[[x_1,\dots,x_n]]$ by an $(x_1,\dots,x_n)$-primary ideal $I\subset S$, and the $k$-algebra structure on $R$ is the same as the $k$-algebra structure on this quotient induced by the inclusion $k\hookrightarrow L_0$.
From now on, we omit the subscript on $L_0$ and simply write $L$.

%\begin{lem}
%We may assume that $L=k$.
%\end{lem}
%
%\begin{proof}
%Write $(R_{L_0})_\infty$ for the ring of Hasse--Schmidt derivations $\HS^\infty_{R/{L_0}}$ and $R_\infty$ for $\HS^\infty_{R/k}$. By \cite[Theorem~2.1]{Vojta} the inclusion $k\hookrightarrow L_0$ induces a map
%$$
%R_\infty \to (R_{L_0})_\infty;
%$$
%by universality of the arc $R\to R_\infty[[t]]$ we obtain the diagram
%$$\begin{tikzcd}%%%CHECK THIS DIAGRAM!!
%(R_\infty/mR_\infty)[[t]] \ar[r] & 
%(R_{L_0})_\infty/m((R_{L_0})_\infty)[[t]]
%\\
%R \ar[u]\ar[ur]
%\end{tikzcd}$$
%If the map 
%$$
%R\to 
%(R_{L_0})_\infty/m((R_{L_0})_\infty)[[t]]
%$$
%is injective (i.e., if the arc-closure of $(0_R)$, viewing $R$ as an $L$-algebra, is trivial) then the map $$R\to (R_\infty/mR_\infty)[[t]]$$ must be injective as well, and thus the arc-closure of $(0_R)$, viewing $R$ as a $k$-algebra, is trivial as well.
%%$$
%%(R_\infty)[[t]]\to (R_{L_0})_\infty [[t]]
%%$$
%\end{proof}

Since $I$ is $(x_1,\dots,x_n)$-primary, there exists $N$ such that $\mN:=(x_1^N,\dots,x_n^N)\subset I$. 
Taking the surjection $S/\mN\to S/I$ and applying the Matlis duality functor $\Hom_S(-,E_S(L))$, 
where $E_S(L)$ is the injective hull of the residue field of $S$,
we obtain an inclusion
$$
\underbrace{\Hom_S(S/I,E_S(L))}_{E_{S/I}(L)} \hookrightarrow \underbrace{\Hom_S(S/\mN,E_S(L))}_{E_{S/\mN}(L)}.
$$
Now, since $S/I$  is assumed to be Gorenstein we have that $E_{S/I}(L)$ is isomorphic as an  $S$-module to $S/I$; likewise for the complete intersection $S/\mN\cong E_{S/\mN}(L)$,
%%IS THIS TRUE? OR IS IT JUST AS AN S/I-MODULE?
so we have an inclusion of $S$-modules 
$$
S/I\hookrightarrow S/\mN;
$$
note that this is in fact an inclusion of $S/\mN$-modules.
Since $S/\mN$ is a graded local $k$-algebra, Theorem 5.8(a) of \cite{DEI} (appearing above as Theorem~\ref{graded}) implies that the zero ideal, viewed as a $S/\mN$-submodule is arc-closed.
But via our comparison lemma (Lemma~\ref{comparison}) we have that 
the arc-closure of $(0_{S/I})$ as an $S/I$-module is contained in the arc-closure of $(0_{S/I})$ as an $S/\mN$-module under the restriction of 
scalars along
$S/(x_1^N,\dots,x_n^M)\to S/I$.
Thus it suffices to show that this latter arc-closure is the zero ideal;
persistence of arc closure for the inclusion of $S/\mN$-modules
$
S/I\hookrightarrow S/\mN
$
gives
$$(0_{S/I})\ac\subset\bigl( 0_{S/\mN }\bigr)\ac=0$$
(with both sides taken as $S/\mN$-modules)
and thus $0_{S/I}\ac = 0$.
\end{proof}

\begin{cor}
Noetherian germs over perfect fields have
the embedded local isomorphism property; likewise if the residue field is already the same as $k$.
%(equivalently, the local isomorphism problem)
\end{cor}

\section{Further questions}
\label{questions}
Despite the triviality  of the arc-closures of ideals in this case, there are related potentially interesting questions:

\begin{rem}
There is another family of jet-theoretic closure operations appearing in \cite{DEI}, the \emph{jet support closures}, defined in terms of the reduced structure of the jet schemes. Explicitly, one can define
a ``reduced'' universal $\ell$-jet or arc via
$$
\bar\l_\ell\colon 
R\to (R_\ell/\m R_\ell)[t]/t^{\ell+1}\to (R_\ell/\m R_\ell)_\red[t]/t^{\ell+1}
$$
or 
$$
\bar\l_\infty\colon 
R\to (R_\infty/\m R_\infty)_\red[[t]].
$$
One then defines $(0)\jsc \ell = \ker \bar \l_\ell$, $(0)\jscn=\bigcap\ker\bar\l_\ell$, and  $(0)\asc =\ker \bar \l_\infty$; 
one can then set $\a \jsc \ell=\pi\inv((0_{R/\a })\jscn)$ and likewise for $a\jscn $ and $a\asc$.
For any ideal $\a $
there are inclusions
$$
\a \ac \subset \a  \jscn \subset \a \asc
$$
and 
$$
\a \jscn \subset \bar \a , 
$$
where $\bar \a  $ is the integral closure of $\a $.
It is shown in 
\cite{DEI} that $\bar \a  = \a \jscn $ for ideals inside a regular ring $R$, but that in a nonregular ring (even for a complete intersection) we may have $\a \jscn \subsetneq \bar \a $.

In contrast to the case for arc-closures, we note the inclusion $\a \jscn \subset \a \asc$ can in fact be proper: for example, if $R=k[x]/x^2$, then one can check explicitly that 
$$
R_\infty/\m R_\infty= k[x_1,x_2,\dots]/(x_1^2,2x_1x_2,2x_1x_3+2x_2^2,\dots),
$$
and that the quotient by the nilradical is just $k$, so that the kernel of $R\to (R_\infty/\m R_\infty)[[t]]$ is the maximal ideal $(x)$, i.e., $(0)\asc = (x)$.
In contrast, one can check that $x\notin (0)\jsc \ell$ for any $\ell$, and thus $(0)\jscn \subsetneq (x)=(0)\asc$.

This suggests that $\a \jscn$ may still be an interesting (and definitely nontrivial) closure operation in the Noetherian case, and provide a geometrically-motivated closure operation tighter than the integral closure in a nonregular ring.
\end{rem}

\begin{rem}
In this paper we introduced arc-closures of submodules to show that arc-closures of ideals are trivial, but it's possible such arc-closures of submodules are nontrivial and interesting. 
In particular, 
basechange $\Omega_{R/k}\mapsto \Omega_{R/k}\otimes_R R_\infty[[t]]$ along the universal arc  
is used in
\cite{differentials}
as part of the description of the K\"ahler differentials of the arc scheme; thus, examination of the map $M\to M\otimes_R R_\infty[[t]]$ may have an interpretation in similar contexts.
\end{rem}

\begin{rem}
%%arc-closures with respect to other Artinian rings
For any Artinian $k$-algebra $A$, there is a scheme of $A$-jets $J_A(X)$, which represents the functor $T\mapsto \Hom(T\times_k A,X)$ on $k$-schemes. Given a $k$-algebra $R$, $J_A(\Spec R)$ will be affine, say $\Spec R_A$; functoriality then gives a universal $A$-jet $\l_A\colon  R\to R_A\otimes_ k A\to R_A/\m R_A\otimes_k A$. 
For more on this construction, see \cite{Mircea}.
Given a complete local ring $(C,\m )$, we can consider a family of Artinian rings $\set{A_\l}$ given by quotients of $C$ by various $\m $-primary ideals $\set{I_\l}$; one can then consider the ideal $\bigcap \ker \l_A$ of $R$, thought of as the $\set{A_\l}$-jet closure of $(0_R)$, and ask if for some suitably chosen $C$ and family of quotients $A_\l$ we obtain an interesting closure operation in this way. 
\end{rem}
\bibliographystyle{alpha}
\bibliography{link}
\end{document}